\documentclass[11pt]{amsart}

\usepackage{amsfonts, amstext, amsmath, amsthm, amscd, amssymb, graphicx, epstopdf}
\usepackage{epsfig, graphics, psfrag, enumerate}
\usepackage{color}
\usepackage{verbatim}

\theoremstyle{plain}
\newtheorem{theorem}{Theorem}[section]

\newtheorem{lemma}[theorem]{Lemma}

\newtheorem{corollary}[theorem]{Corollary}
\theoremstyle{definition}
\newtheorem{definition}[theorem]{Definition}
\newtheorem{example}[theorem]{Example}
\newtheorem{remark}[theorem]{Remark}

\makeatletter
\newcommand{\newreptheorem}[2]{\newtheorem*{rep@#1}{\rep@title}\newenvironment{‌​rep#1}[1]{\def\rep@title{#2 \ref*{##1}}\begin{rep@#1}}{\end{rep@#1}}}
\makeatother

\numberwithin{equation}{section}
\numberwithin{figure}{section}

\newcommand{\Q}{\ensuremath \mathbb{Q}}
\newcommand{\R}{\ensuremath \mathbb{R}}

\newcommand{\Z}{\ensuremath \mathbb{Z}}

\newcommand{\inv}{^{-1}}

\newtheorem*{namedtheorem}{\theoremname}
\newcommand{\theoremname}{testing}
\newenvironment{named}[1]{\renewcommand{\theoremname}{#1}\begin{namedtheorem}}{\end{namedtheorem}}

\begin{document}

\title{Geometric estimates from spanning surfaces}
\author{Stephan D. Burton \and Efstratia Kalfagianni}

\address[]{Department of Mathematics,  Michigan State
University, E Lansing, MI 48824
}

\email[]{burtons8@math.msu.edu}

\address[]{Department of Mathematics, Michigan State
University, E Lansing, MI 48824} \ \ \
\email[]{kalfagia@math.msu.edu}

\thanks{Supported by NSF Grants DMS-1105843, DMS-1404754 and DMS - 1708249.}
\begin{abstract} We derive bounds on the length of the meridian and the cusp volume
of hyperbolic knots in terms of the topology of essential surfaces spanned by the knot.
We provide an algorithmically  checkable criterion that guarantees that the meridian length of a hyperbolic  knot is below a given bound.
As  applications we find knot diagrammatic upper bounds on  the meridian length and the cusp volume of hyperbolic adequate knots and we obtain new large families of knots with
meridian lengths bounded above by four. We also discuss applications of our results to Dehn surgery.
\end{abstract}
\bigskip

\bigskip
\maketitle


\section{Introduction}

An important goal  in knot theory is to relate the geometry of knot complements to topological and combinatorial quantities and invariants of knots.
In this paper we derive bounds of slope lengths on the maximal  cusp and of the cusp volume of hyperbolic knots in terms of the topology of essential surfaces spanned by the knots.
Our results are partly motivated by the open question of whether there exist hyperbolic knots in $S^3$ whose meridian length exceeds four.
We show that there is an algorithmically checkable criterion to decide whether a hyperbolic knot has meridian length less than a given bound, and we use it to we obtain  large families of knots with
meridian lengths bounded above by four.
Our results are particularly interesting in the case of knots that  project on closed embedded surfaces in an alternating fashion and  admit essential checkerboard surfaces. In this case our bounds
are purely combinatorial and can be read directly from a knot diagram. 
We also discuss applications of our results to Dehn surgery.

Given a  hyperbolic knot $K$ in $S^3$, there is a well-defined notion of a maximal cusp $C$ of the complement
$M=S^3\setminus K$. The interior of $C$ is neighborhood of the missing $K$ and the boundary $\partial C$
is a  torus   that inherits a Euclidean structure from the hyperbolic metric. 
Each slope  $\sigma$ on $\partial C$   has a unique geodesic representative. 
The length of  $\sigma$, denoted by $\ell(\sigma)$, is the length of its geodesic representative.  By Motsow-Prasad rigidity, these lengths are topological  invariants of $K$.

By abusing notation and terminology we will also refer to $\partial C$ as the boundary of $M$. We will sometimes use the alternative notation $\partial M$.
 For a slope $\sigma$ on $\partial M$ let $M(\sigma)$ denote the 3-manifold
obtained by Dehn filling $M$ along $\sigma$. By the knot complement theorem of Gordon and Luecke \cite{GoLu}, there is a unique slope $\mu$, called the meridian of $K$,  such that  $M(\mu)$ is $S^3$. 
A $\lambda$-curve of $K$ is a  slope on $\partial M$ that intersects $\mu$ exactly once and a spanning surface of $K$ is a   properly embedded surface in $M$ whose boundary is a $\lambda$-curve.

\begin{theorem}\label{meridiancriterion} Let $K$ be a hyperbolic knot with  meridian length $\ell(\mu)$. Suppose that $K$ admits essential spanning surfaces 
 $S_1$ and $S_2$ such that 
\begin{equation} \label{eq:Criterion}
|\chi(S_1)| + |\chi(S_2)| \leq {\frac{b}{6}\cdot i(\partial S_1, \partial S_2) },
\end{equation}
 where $b$ is a positive real number and
 $ i(\partial S_1, \partial S_2) $  the minimal intersection number of $\partial S_1, \partial S_2$ on $\partial M$.
 Then the meridian length satisfies $\ell (\mu)  \leq  b$. 
 
 Moreover, given a hyperbolic knot $K$ and  $b>0$, there is an algorithm to determine if there are essential surfaces $S_1$ and $S_2$ satisfying (\ref{eq:Criterion}). 
\end{theorem}

A  slope $\sigma$ on $\partial M$ is called \emph{exceptional} if the 3-manifold $M(\sigma)$ is not hyperbolic. The Gromov-Thurston ``$2\pi $-theorem" \cite{2pi} asserts that if $\ell(\sigma)>2\pi$ then $M(\sigma)$ admits a Riemannian metric of negative curvature. This combined with the proof of Thurston's  geometrization conjecture \cite{Geotheo} implies that  actually $M(\sigma)$ is hyperbolic.
The work of Agol \cite{Pleated} and Lackenby \cite{LackenbyWordHyperbolic}, that
 has improved $2\pi$ to $6$, asserts that exceptional slopes must have length less than or equal to six.  Examples of exceptional slopes with length six are given
 in \cite{Pleated} and in  \cite{sharp}.
 Since the meridian curve of every  hyperbolic knot in $S^3$ is an exceptional slope, we have $\ell(\mu)\leq 6$.
 The work of Adams, Colestock, Fowler, Gillam, and Katerman \cite{AdamsCuspSizeBounds} shows that  that $\ell(\mu)<6$. Examples of knots whose meridian length approach four from below are given in \cite{Pleated}
 and by Purcell in \cite{ PurcellSlopeLengths}.
 An open conjecture in the area is that for all hyperbolic knots in $S^3$ we should have $\ell(\mu)\leq 4$.

 Theorem \ref{meridiancriterion} provides a 
 criterion for checking algorithmically whether  a given knot satisfies this conjecture. Indeed, given a hyperbolic knot $K$ there is an algorithm using normal surface theory to decide whether $K$ admits  essential spanning surfaces $S_1,S_2$ 
such that
$$|\chi(S_1)| +|\chi(S_2)| \leq {\frac{4}{6}\cdot i(\partial S_1, \partial S_2) },$$
and thus whether $\ell(\mu)\leq 4$.

Next we will discuss applications of Theorem \ref{meridiancriterion}.
As a warm up example, we first mention  the hyperbolic  3-pretzel knots $P(a, -b, -c)$ with $a, b, c>1$ and all odd. For these knots Theorem \ref{meridiancriterion} applies
to give $\ell(\mu)\leq 3$. See example \ref{pretzel} for details and for generalizations.
\vskip 0.1in

\subsection{Knots with essential checkerboard surfaces.} Theorem \ref{meridiancriterion}  can be applied to knots that admit alternating projections on closed surfaces so that they define essential checkerboard surfaces.  
A large such class of knots is the class of \emph{adequate knots}, that admit alternating projections with essential checkerboard surfaces on certain \emph{Turaev surfaces}.
In this case, we have the following theorem, where the terms involved are defined in detail in  Sections \ref{hyperbolic} and  \ref{Turaevdef}.

\begin{theorem}\label{thm:MeridianBound}
Let $K$ be an adequate hyperbolic knot in $S^3$ with crossing number $c=c(K)$  and Turaev genus $g_T$. Let $C$ denote the maximal cusp of $S^3\setminus K$ and let
${\rm Area}(\partial C)$ denote the cusp area. Finally let
$\ell(\mu)$ and $\ell(\lambda)$ denote the length of the meridian and the  shortest $\lambda$-curve of $K$.
Then we have

\begin{enumerate}
 \item $ \displaystyle \ell(\mu) \leq 3 + \frac{6g_T-6}{c} $
 \item $\ell(\lambda) \leq 3c + 6g_T - 6$
 \item ${\rm Area}(\partial C) \leq {9c} \left(1 + \dfrac{ 2g_T-2}{c}\right)^2$
\end{enumerate}
\end{theorem} 

A knot is alternating precisely when 
 $g_T = 0$.  In this case, the bounds of Theorem \ref{thm:MeridianBound} agree with the bounds of
\cite{AdamsCuspSizeBounds}. The technique of the  proof  of Theorems \ref{meridiancriterion} and  \ref{thm:MeridianBound}, as well as the proof of results in   \cite{AdamsCuspSizeBounds},
 is reminiscent of arguments with pleated surfaces that led to the proof of the ``6-Theorem" \cite{Pleated, LackenbyWordHyperbolic}. The algorithm  for
 checking criterion (\ref{eq:Criterion}) involves normal surface theory and in particular the work of 
  Jaco and Sedgwick \cite{JacoDecisionProblems}.

Similar estimates to those of Theorem  \ref{thm:MeridianBound} below should should work for the class of  \emph{weakly alternating knots} studied in \cite{Ozawa1}.  
See Remark \ref{generalize}.
\vskip 0.02in

\subsection{Knots with meridian length bounded by four} As mentioned  earlier, it has been conjectured that the meridian length of every hyperbolic knot in $S^3$ is at most four.
The conjecture is known for several classes of knots.
Adams \cite{AdamsTwoGenerators} showed that the meridian of a 2-bridge hyperbolic knot has length less than 2. By  \cite{AdamsCuspSizeBounds} when  $K$ is an alternating hyperbolic knot
then  $\ell(\mu) < 3$. 
Agol \cite{Pleated} found families of knots whose meridian lengths approach four from below and  Purcell \cite{PurcellSlopeLengths} generalized his construction
to construct families of knots whose meridian length approach four from below. She also showed that ``highly twisted" knots have meridian lengths  less than four.
 Our results in this paper allow us to verify the meridian length conjecture for additional broad classes of hyperbolic knots. Again restricting to adequate knots for simplicity, we give two sample results.
Notice that,  by Theorem \ref{thm:MeridianBound}, if $c\geq 6g_T-6$ then $\ell(\mu)\leq 4$. Thus, for every Turaev genus
there can be at most finitely many adequate knots with  $\ell(\mu)>4$. In particular if $g_T \leq 3$, then $\ell(\mu)\leq 4$ unless $c\leq 12$. Since the knots up to 12 crossings are known to have meridian lengths less that two \cite{Knotinfo}, in fact, we have:
\begin{corollary}\label{finite}
Given $g_T>0$, there can be at most finitely many hyperbolic  adequate knots of Turaev genus $g_T$ and with $\ell(\mu)>4$.  
In particular, if $K$ is a hyperbolic adequate knot with $g_T\leq 3$, then we have $\ell(\mu)<4$.
 \end{corollary}

Note that for  $g_T=1$, we actually get $\ell(\mu)\leq 3$. Knot diagrams of Turaev genus one were recently classified \cite {AL, Kim}.
The  case of adequate diagrams  includes Conway sums of strongly alternating tangles (see \cite{LickorishSomeLinks}). We therefore have that if a knot $K$ is a Conway sum of strongly alternating links, then the length of the meridian of $K$ is less or equal to three.

Another instance where our length bounds work well is to show that  knots admitting diagrams with large ratio of crossings to twist regions have small meridian length. We have the following result which in particular applies to 
closed positive braids. See Corollary \ref{braids}.

\begin{theorem} \label{twist}
Let $K$ be a hyperbolic knot with an adequate diagram with $c$ crossings and $t$ twist regions. Then we have
$$\ell(\mu)\leq 3 +  \frac{3t}{c}- \frac{6}{c}.$$
 In particular if $c\geq 3t$ then we have $\ell(\mu) < 4$.

\end{theorem}

\subsection{Slope length bounds,  Dehn filling and volume} Let $K$ be a hyperbolic knot with maximal cusp $C$ and slopes 
 $\sigma, \sigma'$  on $\partial C$.
Calculating area in Euclidean geometry on $\partial C$  (see for example the proof of \cite[Theorem 8.1]{Pleated}),  we have

\begin{equation}
\ell(\sigma) \ell(\sigma') \geq \text{Area}(\partial C) \Delta(\sigma, \sigma'),
\label{elem0}
\end{equation}

\noindent  where $\Delta(\sigma, \sigma')$ denotes the absolute value of the intersection number of $\sigma, \sigma'$.
Work of Cao and Meyerhoff \cite[Proposition 5.8] {MeyerhoffMinimumVolume} shows that $\text{Area}(\partial C) \geq 3.35$. Given an adequate hyperbolic knot
$K$, we  will apply  (\ref{elem0}) for $\sigma'=\mu$.
Using the upper bound for $\ell(\mu)$ from Theorem \ref{thm:MeridianBound}, we have

\begin{equation}
\ell(\sigma) > \dfrac{3.35 \Delta(\mu, \sigma) c}{3c + 6g_T - 6}= \dfrac{3.35}{3} \cdot \dfrac{ \Delta(\mu, \sigma) }{1+\delta}, 
\label{elem2}
\end{equation}

\noindent  where 
$\delta= \dfrac{2g_T - 2}{c}$. We note that $\delta$ is an  invariant of $K$ that can be  calculated from any adequate diagram (see Theorem \ref{abe}).
Now   (\ref{elem2})  implies that  if
$$\Delta(\mu, \sigma) > \dfrac{18}{3.35} \left( 1 +\delta \right)> 5.37\left(1+\delta \right) ,$$
then $\ell(\sigma) > 6$ and thus $\sigma$ cannot be an exceptional slope.

Note that if $\sigma$ is a slope represented by $ {p}/{q}\in \Q$ in $H_1(\partial C)$ then 
$\Delta(\mu, \sigma) =|q|$.  Hence if $|q|>6 (1+\delta)$,  inequality  (\ref{elem2}) implies that $\ell(\sigma)>  \dfrac{3.35}{3}\cdot  6> 2\pi.$
In this case, we may apply a result of Futer, Kalfagianni and Purcell  \cite[Theorem 1.1]{fkp:filling} to estimate the change of  volume under Dehn filling of adequate knots. We have the following.

\begin{theorem}\label{surgery}
Let $K$ be a hyperbolic adequate  knot and let
$\delta$ be as above.  If $|q|\geq 6(1+\delta)$, then  the 3-manifold $N$ obtained by $ {p}/{q}$ surgery along $K$
 is  hyperbolic and the volume satisfies the following
 $${\rm vol}(S^3\setminus K)\ >\ {\rm vol}(N) \ \geq \
 \left(1-\frac{36(1+\delta)^2}{q^2}\right)^{3/2}{\rm vol}(S^3\setminus K). $$
\end{theorem}
The assertion that $N$ is hyperbolic follows immediately from above discussion. The left hand side inequality is due to the result of Thurston  that the hyperbolic volume drops under Dehn filling \cite{thurston:notes}.
The right hand side follows by \cite[Theorem 1.1]{fkp:filling}.

Theorem 5.14 of \cite{GutsBook}, and its corollaries,  give diagrammatic bounds for $ {\rm vol}(S^3\setminus K)$ in terms any adequate diagram of $K$. This combined with 
Theorem \ref{surgery} implies that the  volume of $N$ can be estimated from any adequate diagram of $K$. For example, Montesinos knots with a reduced 
diagrams that contains at least two positive
tangles and at least two negative tangles are adequate and have $\delta\leq 0$.
Combining Theorem \ref{surgery}
with
\cite[Theorem 9.12]{GutsBook} and \cite[Theorem 1.2]{FP}  we have the following.

\begin{corollary}Let $K \subset S^3$ be a Montesinos link with a reduced 
diagram $D(K)$ that contains at least two positive
tangles and at least two negative tangles.
 If $|q|\geq 6$, then  the 3-manifold $N$ obtained by $ {p}/{q}$ surgery along $K$
 is  hyperbolic and we have
 $$ 2 v_8 \, t \ >\ {\rm vol}(N) \ \geq \
 \left(1-\frac{36}{q^2}\right)^{3/2} \frac{v_8}{4} \left( t- 9\right),$$
\noindent where  $t=t(D)$ is the twist number of $D(K)$, and 
$v_8 = 3.6638...$ is the volume of a regular ideal octahedron.
\end{corollary}

\vskip 0.03in

\subsection{Organization} In Section \ref{hyperbolic} we recall the hyperbolic geometry terminology we need for this paper, and the results and facts about pleated surfaces we will use.
In Section \ref{Turaevdef} we recall results and terminology about adequate knots and their Turaev surfaces we need in subsequent sections. In Section \ref{proofs} we derive the bound of the meridian length in Theorem \ref{meridiancriterion} and corresponding bounds for the length of the  shortest $\lambda$-curve and cusp volume. See Theorem \ref{thm:GeneralMeridianBound}.
Then we prove Theorem  \ref{thm:MeridianBound}   and its corollaries. In Section \ref{algorithm} we show that given $K$ and $b>0$ there is an algorithm
 which determines if there are essential spanning surfaces $S_1$ and $S_2$ satisfying inequality (\ref{eq:Criterion}).
This completes the proof of Theorem \ref{meridiancriterion}.

\vskip 0.03in

\subsection{ Acknowedgement} We thank Colin Adams, Dave Futer, Cameron Gordon, and Jessica Purcell for discussions, comments and interest in this work.


\section{Hyperbolic Geometry Tools}\label{hyperbolic}

In this section we review some notions and results in hyperbolic geometry that we will need in this paper. Let $M$ be a 3-manifold whose interior has a hyperbolic structure of finite volume.
Let $\mathbb{H}^3$ denote the 3-dimensional hyperbolic space model and let $\rho:\mathbb{H}^3 \to M$ be the covering map. 
Then $M$ has ends of the form $T^2 \times [1, \infty)$, where $T^2$ denotes a torus. Each end is geometrically realized as the image of some  $C = \rho(H)$ of some horoball  $H \in \mathbb{H}^3$. The pre-image $\rho\inv(C)$ is a collection of horoballs in $\mathbb{H}^3$.  For each end there is a 1-parameter cusp family obtained by expanding 
 the horoballs of $\rho\inv(C)$ while keeping the same limiting points on the sphere at infinity. 
 By expanding the cusps until in the pre-image $\rho\inv(C)$ each horosphere is tangent to another, we obtain a choice of \emph{maximal cusps}. The choice depends on the the horoballs $H$.
 If $M$ has a single end then there is a well defined maximal cusp referred to as the \emph{the maximal cusp} of $M$. 
  
  \begin{definition} Given a hyperbolic knot $K$ the complement $M=S^3\setminus K$ is a hyperbolic 3-manifold with one end.  The \emph{cusp} of $K$, denoted by $C$,  is the maximal cusp of $M$.
  The boundary  $R_H$  of the horoball $H$ is a horosphere and the boundary 
  of $C$, denoted by $\partial C$, inherits a Euclidean structure from $\rho | R_H: R_H  \longrightarrow  \partial C$. The \emph{ cusp area} of $K$, denoted by ${\rm Area}(\partial C)$ is the  Euclidean area of $\partial C$ and \emph{the cusp volume} of $K$, denoted by ${\rm Vol}( C)$ is the volume of $C$. Note that we have  ${\rm Area}(\partial C)=2\,{\rm Vol}( C)$.
   
   The length of the meridian of $M=S^3\setminus K$, denoted by  $\ell(\mu)$, is defined to be the Euclidean length of  the geodesic representative on $\partial C$ of a meridian curve $\mu$ of $K$.
 Recall that a $\lambda$-curve on $\partial C$ is one that intersects the meridian exactly once.
The length of a geodesic representative of a shortest $\lambda$-curve on $\partial C$ will be denoted by $\ell(\lambda)$. Note that there may be multiple shortest
$\lambda$-curves. Nevertheless, they all have the same length and we will refer to it as  the length of \emph{the}  shortest $\lambda$-curve on $\partial C$.

The cusp area is bounded above by $\ell(\mu)\ell(\lambda)$, where equality holds if  $\mu$ and $\lambda$ are perpendicular.
\end{definition}

 An embedded surface (possibly non-orientable) $S\subset M$, with 
each component of $\partial S$ embedded on 
$\partial C$  is called \emph{essential} 
if the oriented double of $S$ is  \emph{incompressible}  and  \emph{$\partial$-incompressible}. See, for example, \cite[Definition 1.3]{GutsBook}.

Consider  a (possibly non-connected)  surface $S$ (possibly with boundary)  and a singular continuous map $f: S \longrightarrow M$ that embeds each component of $\partial S$ 
in $\partial C$. We will say that $f$ is \emph{homotopically-essential} if (i) the image of no essential simple closed loop on $S$ is homotopically trivial in $M$;
and (ii) the image of no essential embedded arc on $S$ can be homotoped (relatively its endpoints) on $\partial C$.
 If $S\subset M$ is an essential (i.e. $\pi_1$-injective) embedded surface, the inclusion map is homotopically-essential.

Next we  recall 
Thurston's notion of \emph{pleated surface}. See Thurston's notes 
\cite{thurston:notes} or the exposition by Canary, Epstein and Green \cite{notesnotes} for more details.

\begin{definition} A singular continuous map $f:(S, \ \partial S) \longrightarrow (M, \partial C)$ is called \emph{pleated} if  the following are true:
(i) the components of $\partial S$ map to geodesics on
$\partial C$; (ii) the interior of $S$, denoted by ${\rm int}(S)$, is triangulated so that each triangle maps under $f$ to a subset of $M$ that lifts to an ideal hyperbolic geodesic triangle in  $\mathbb{H}^3$;
and (iii) the 1-skeleton of the triangulation forms a lamination on $S$.

Given a pleated map $f$ we may pull-back the path metric from $M$ by $f$ to obtain a hyperbolic metric on ${\rm int}(S)$, where the 1-skeleton lamination is geodesic.
\end{definition}

We need the following lemma. For a proof the reader is referred to \cite{notesnotes, thurston:notes} or to \cite[Lemma 4.1]{Pleated}.

\begin{lemma} \label{pleatf} Let $M=S^3\setminus K$ be a hyperbolic knot complement and let $S$ be a surface with boundary and $\chi(S)<0$.
Let $f: (S, \ \partial S) \longrightarrow (M, \partial C)$  be a homotopically essential map and suppose that each component of $\partial S$ is mapped to a geodesic in $\partial C$.
Then there is a pleated map $g: (S, \ \partial S) \longrightarrow (M, \partial C)$, such that $g|{\rm int}(S)$ is homotopic to $f|{\rm int}(S)$ and a hyperbolic metric on $S$ so that
$g|\partial S$ is an isometry.

\end{lemma}

Let $M=S^3\setminus K$ be a hyperbolic knot complement with maximal cusp $C$ and let $f: (S, \ \partial S) \longrightarrow (M, \partial C)$  be a homotopically essential map that is pleated.
 In this paper we are interested in the case that $S$ is the disjoint
union of spanning surfaces of $K$. Suppose that $\partial S$ has $s$ components.  The geometry of $f(S)\cap C$ can be understood using arguments of  
\cite[Theorem 5.1]{Pleated} and  \cite[Lemma 3.3]{LackenbyWordHyperbolic}. By  the argument in the proof of \cite[Theorem 5.1]{Pleated}, we can find disjoint horocusp neighborhoods 
$H=\cup_{i=1}^sH_i$ of $S$, such that $f(H_i)\subset C$,  $\ell(\partial H_i)={\rm{Area}}(H_i)$ and such that $\ell(\partial H_i)$ is at least as big as the length of $f(\partial H_i)$ measured on $C$. Thus we have 
$$\ell_C(S)\leq \sum_{i=1}^{s}\ell(\partial H_i)={\rm{Area}}(H),$$
where
 $\ell_C(S)$ denotes
the total length of the intersection curves in $f(S) \cap \partial C$.
Since, for all $i\neq j$, we have  $H_i\cap H_j= \emptyset$, a result of B\"or\"oczky \cite{density}
on horocycle packings in the hyperbolic place applies. Using this result
one obtains $$\sum_{i=1}^{s}{\rm{Area}} (H_i)\leq \frac{6}{2\pi} {\rm Area}(S)=\frac{6}{2\pi} (2\pi  | \chi(S)|),$$
where the last equation follows by  the Gauss-Bonnet theorem.
The above inequality is also proven in 
  \cite[Lemma 3.3]{LackenbyWordHyperbolic}.
Combining all these leads to the following Theorem which is a special case of
 \cite[Theorem 5.1]{Pleated} and  \cite[Lemma 3.3]{LackenbyWordHyperbolic}.

\begin{theorem}\label{thm:pleated} Let $M=S^3\setminus K$ be a hyperbolic knot complement with maximal cusp $C$. Suppose that
$f: (S, \ \partial S) \longrightarrow (M, \partial C)$  is a homotopically essential map that is pleated and let $\ell_C(S)$ denote
the total length of the intersection curves in $f(S) \cap \partial C$. Then we have
$$\ell_C(S)\ \leq \  6 | \chi(S)|.$$
\end{theorem}


\section{ Knots with essential  checkerboard surfaces}\label{Turaevdef}
A setting where pairs of spanning surfaces of knots occur naturally is the checkerboard surfaces of knot projections on surfaces. We are interested in knots with projections where the checkerboard surfaces
are essential in the knot complement. A well-known class of knots admitting such surfaces  are knots that admit alternating projections on a 2-sphere (alternating knots).
Generalizations include the class of \emph{adequate knots} that arose in the study of Jones type invariants. Below we will review some terminology and results about such knots that we need in this paper.
\vskip 0.03in


\subsection{Adequate diagrams and knots}
Let $D$ be a diagram for a knot $K$. At each crossing of the diagram $D$ one may resolve the crossing in one of two ways: the $A$-resolution and the $B$-resolution as depicted in Figure \ref{resolve}. A choice of resolutions of crossings of $D$ is called a state $\sigma$.  The result of applying the state $\sigma$ to $D$, denoted $s_\sigma(D)$, is a collection of disjoint circles called state circles. One may then form the state graph $G_\sigma$ where vertices correspond to state circles of $s_\sigma(D)$ and and edges correspond to former crossings in $D$.

\begin{definition}\label{defi:adequate} A diagram $D$ is called \emph{adequate} if the state graphs  of the all-$A$ and all-$B$-resolutions  have no 1-edge loops. A knot is called  \emph{adequate} 
if it has an adequate diagram.
\end{definition}

Given a diagram $D$ of a knot $K$, one may form a surface $S_A$ as follows. The state circles of  the   all-$A$ resolution of $D$
bound disks  on the projection plane. Isotope these disks slightly off the projection plane so they become disjoint.
For each crossing of $D$, attach a half-twisted band so that the resulting surface $S_A$ has boundary $\partial S_A = K$. One may form the surface $S_B$ similarly. 
See Figure \ref{resolve}.

\begin{figure}
\includegraphics[scale=.8]{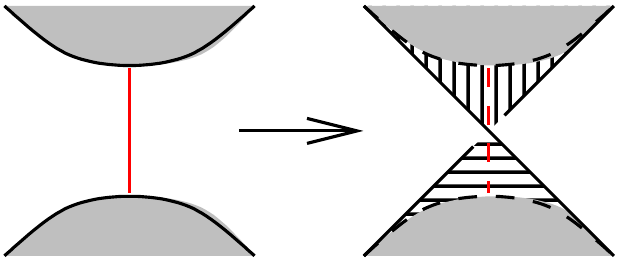}
\hspace{2cm}
\includegraphics[scale=.8]{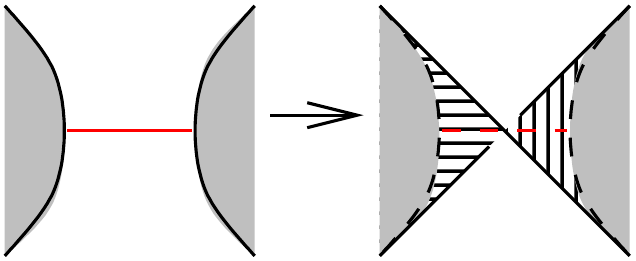}
\caption{The two resolutions of a crossing, the arcs recording them, and their contribution to state surfaces. The left frame depicts the $A$-resolution; the right depicts the $B$-resolution.} 
\label{resolve}

\end{figure}

The following theorem is due to Ozawa \cite{Ozawa}. A different proof  is given by  
 Futer, Kalfagianni, and Purcell \cite[Theorem 3.19]{GutsBook}.
\begin{theorem}\label{essential}
Let $D(K)$ be an adequate link diagram of a knot $K$. Then the all-$A$ state  and the all-$B$ state surfaces corresponding to $D(K)$ are essential in $S^3 \backslash K$.
\end{theorem}

\subsection{Turaev Surfaces}

The \emph{Turaev genus} of a knot diagram $D=D(K)$ with $c$ crossings is  defined by
$g_T(D)=(2-v_A-v_B+c)/2$, where $v_A, v_B$ denotes  the number of the state circles in the all-$A$ and all-$B$ resolutions of $D$ respectively.
The Turaev genus of a knot $K$ is defined by
$$
g_T(K)= {\rm min } \left\{ g_T(D)\ | \   D=D(K) \right \}.
$$
The genus $g_T(D)$ is the genus of the {\emph Turaev surface} $F(D)$  corresponding to $D$. This surface is constructed as follows.
Let $\Gamma \subset S^2$ be the planar, 4--valent graph defined by  $D$.  Thicken the (compactified) projection plane to $S^2 \times
[- 1, 1]$, so that $\Gamma$ lies in $S^2 \times \{0\}$. Outside a
neighborhood of the vertices (crossings),  $\Gamma \times [- 1, 1]$ will be part of $F(D)$. 

\begin{figure}
\includegraphics[scale=0.65]{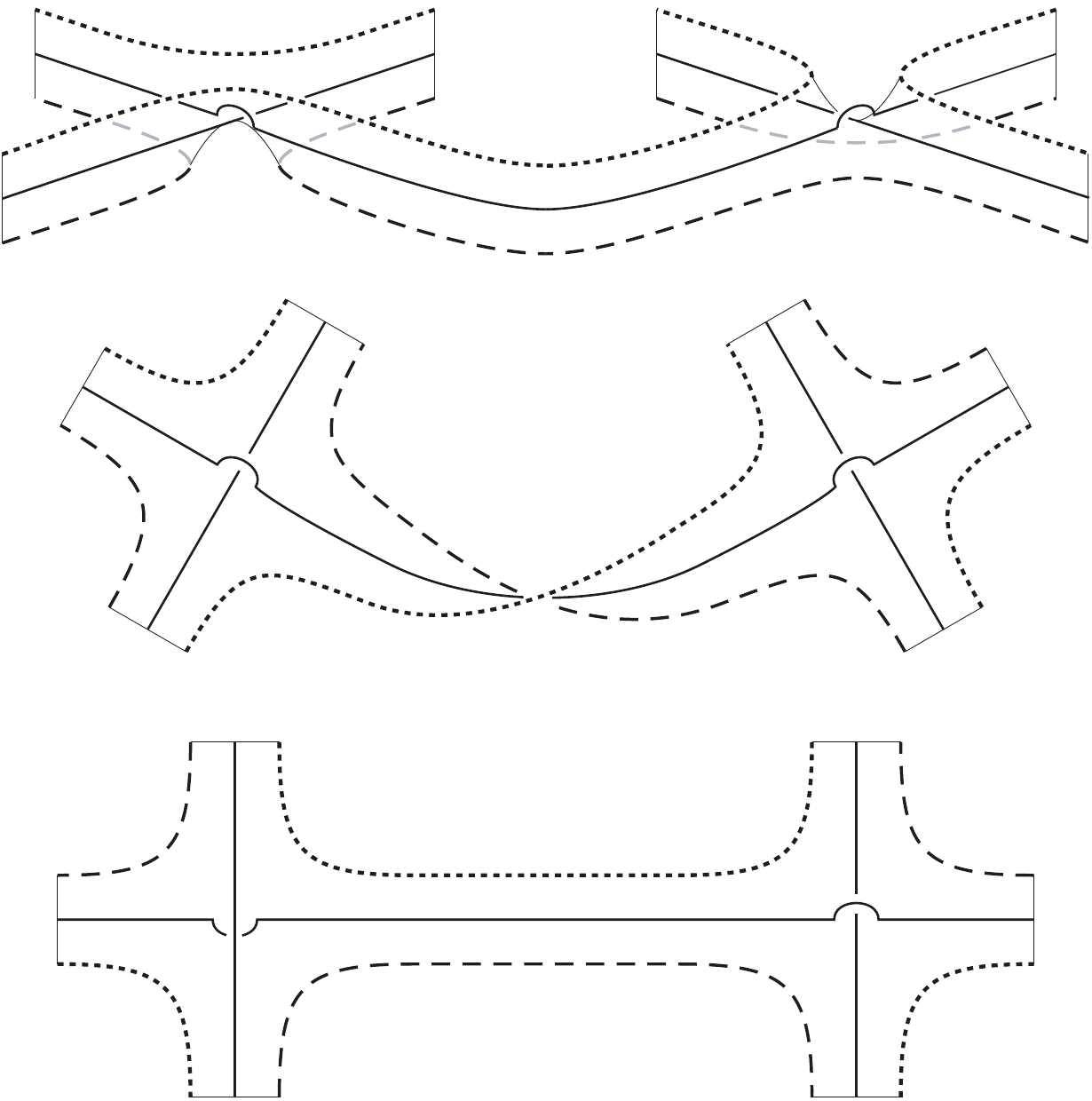}
\caption{Saddles of $F(D)$ corresponding to  two  successive over-crossing of $D$.  The  third picture illustrates how  $D$ is
is alternating on $F(D)$. The figure is taken from \cite{GraphsOnSurfaces}.}
\label{fig:TuraevAlternating}
\end{figure}

In the neighborhood of
each vertex, we insert a saddle, positioned so that the boundary
circles on $S^2 \times \{1\}$ are the
components
of the $A$--resolution and the boundary circles on $S^2
\times \{- 1\}$ are the components of  the  $B$--resolution.

The following is proved in \cite{GraphsOnSurfaces}.

\begin{lemma}  \label{turaevsurface} The Turaev surface $F(D)$ has the following properties:

 (i) It  is a  Heegaard surface of ${S}^{3}$.
 
  (ii) $D$ is alternating  
on $F(D)$; in particular $D$ is an alternating diagram if and only if $g_T(F(D))=0$. See Figure \ref{fig:TuraevAlternating}.

(iii)  The 4-valent graph underlying $D$ defines a cellulation of $F(D)$ for which the 2-cells can be colored in a 
checkerboard fashion.

(iv) The checkerboard surfaces defined by $D$ on $F(D)$ are the state surfaces $S_A$ and $S_B$.
\end{lemma}

We note that an adequate diagram realizes the crossing number of the knot; thus it is a knot invariant.
The following  result of Abe \cite[Theorem 3.2]{AbeTuraevGenus} shows that the same is true for the Turaev genus.

\begin{theorem} \label{abe} Suppose that $D$ is an adequate diagram of a knot $K$. Then,
$$2g_T(K)=2g_T(D)=2-v_A(D)-v_B(D)+c(D).$$ \qed
\end{theorem}

\vskip 0.2in

\section{Lengths of Curves on the Maximal Cusp Boundary}\label{proofs}

In this section, we prove the main results of this paper. We begin by giving a general bound for lengths of curves in the boundary of a maximal cusp neighborhood of a hyperbolic knot. We then apply this bound to the special cases of adequate knots and three-string pretzel knots. 
\begin{theorem}\label{thm:GeneralMeridianBound}
Let $K$ be a hyperbolic knot  with maximal cusp $C$.
Suppose that
 $S_1$ and $S_2$ are essential spanning surfaces in $M=S^3 \setminus K$ and let $i(\partial S_1, \partial S_2)\neq 0$ denote the minimal intersection number of $\partial S_1, \partial S_2$ in $\partial C$. 
 Let  $\ell(\mu)$ and   $\ell(\lambda)$  denote the length of the meridian and the shortest $\lambda$-curve of $K$, respectively. Then we have:
\begin{enumerate}
\vskip 0.02in

\item $\ell(\mu)\leq  \dfrac{6(|\chi(S_1)| + |\chi(S_2)|)}{i(\partial S_1, \partial S_2)} $
\vskip 0.02in

\item $\ell(\lambda) \leq 3( | \chi(S_1) | + |\chi(S_2)|)$
\vskip 0.02in

\item ${\rm Area}(\partial C) \leq  18 \dfrac{(|\chi(S_1)| + |\chi(S_2)|)^2}{i(\partial S_1, \partial S_2)}$
\end{enumerate} 
\end{theorem}
\begin{proof}

Consider $S$ to be the disjoint union of $S_1, S_2$, and let $f: S\longrightarrow M$, where $f(S)$ is the union of $S_1, S_2$ in the complement of $K$.
Since $f|S_i$ is an embedding for $i=1,2$, and each $S_i$ is essential, $f$ is a homotopically essential map. Hence, by Lemma \ref{pleatf}, we may pleat $f$ and then apply Theorem
\ref{thm:pleated}.   With the notation as in that theorem we have
 $$\ell_C(S)\ \leq \  6 | \chi(S)|,$$
where  $\ell_C(S)$ is the total length of the curves $f(S) \cap \partial C$. 

To find bounds of this total length, we orient $\partial S_1, \partial S_2$ and $\mu$ so that $\partial S_1, \partial S_2$ have opposite algebraic intersection numbers with $\mu$.
Let $[\partial S_1], [\partial S_2]$, and $[\mu]$ denote their  classes  in $\pi_1(\partial C)=H_1(\partial C)$.
Since $S_1$ is a spanning surface, we know that $[\partial S_1]$ and $[\mu]$ generate $\pi_1(\partial C)$.

 Recall the covering  $\pi:=\rho | R_H: R_H  \longrightarrow  \partial C$, where $R_H$ is the boundary of a horoball at infinity, say  $H\subset  \cup \rho^{-1} (C)$.
 To fix ideas, assume that  $\partial S_1$ lifts to the horizontal lines $\pi\inv(\partial S_1) = \{(x, n): x\in \R\}$ for each $n \in \Z$ and where $\mu$ lifts to the vertical lines $\pi \inv(\mu) = \{(n, y): y \in \R\}$ for each $n \in \Z$. We may apply a homotopy to $\mu$ so that $\partial S_1 \cap \partial S_2 \cap \mu = \{x_0\}$, where  $\pi\inv(x_0) = \Z^2$. 

Since $[\partial S_1]$ and $[\mu] $ generate $\pi_1(\partial C)$, we can write $[\partial S_2] =  \alpha [\mu] +\beta [\partial S_1]$ for some $\alpha, \beta \in \Z$. The fact that $S_2$ is a spanning surface implies   
$|\beta |= 1$ and $|\alpha|= i(\partial S_1,  \partial S_2)$. Therefore $[\partial S_2]$ can be represented as a curve which lifts to the segment $\{(x, \alpha x) : x \in [0,1]\} \subset \R^2=R_H$. 

\begin{figure}
\def\svgwidth{120pt}
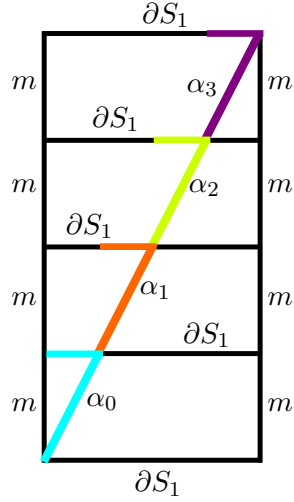
\caption{The arcs $\alpha_k$ are each homotopic to the meridian, and their union projects to $\partial S_1 \cup \partial S_2$.}
\label{fig:count_meridians}
\end{figure}
 
The collection of arcs $$\alpha_k = \{(x, \alpha x) : x \in [k/\alpha ,(k+1)/\alpha]\} \cup \{(x, k+1) : x \in [k/\alpha, (k+1)/\alpha]\}$$ for $k = 0, 1, \hdots, \alpha-1$ is mapped to $\partial S_1 \cup \partial S_2$ by $\pi$. Moreover, each $\pi(\alpha_k)$ is a loop in $\partial C$ homotopic to a meridian. See Figure \ref{fig:count_meridians}, where  each $\alpha_k$ is indicated in a different color.
Therefore $\partial S_1 \cup \partial S_2$ can be decomposed into a collection of simple closed curves that contain $|\alpha|$ meridians. Hence we obtain
$$i(\partial S_1, \partial S_2) \ell(\mu)\leq   \ell_C(S) \leq 6 |\chi(S_1)| +6 |\chi(S_2)|.$$

The decomposition of  $\partial S_1 \cup \partial S_2$ described above can be also seen by resolving all the intersections of $\partial S_1, \partial S_2$ in a way consistent with the orientations chosen above. 

To prove part (2), 
consider $\partial S_1$ and $\partial S_2$  oriented as above in $\partial C$. By resolving the crossings of $\partial S_1$ with $\partial S_2$ in a manner  not consistent with the orientations of $\partial S_1$ and $\partial S_2$, one obtains two $\ell$-curves in $\partial C$. Thus  $2 \ell(\lambda)\leq \ell_C(S)$ and Theorem \ref{thm:pleated} now implies that 

$$2 \ell(\lambda)< 6|\chi(S_1)| + 6 |\chi(S_2)|.$$

To prove part (3), observe that ${\rm Area}(\partial C) \leq \ell(\mu)\ell(\lambda)$.
\end{proof}

As an example, we apply Theorem \ref{thm:GeneralMeridianBound} to 3-string pretzel knots. Note that non-alternating 3-string pretzel knots are not adequate as it follows from the work  of Lee and van der Veen \cite{LeeVeen}.

\begin{example}\label{pretzel}
Let $K$ be the pretzel knot $P(a, -b, -c)$ with $a, b, c$ all positive and odd. The standard 3-pretzel diagram of $K$ is $A$-adequate.
Hence the corresponding all-$A$ state surface $S_A$ is essential in the complement of $K$.
Moreover, the 3-pretzel surface $S_P$ is a minimum genus Seifert surface for $K$ and thus also essential. 
The boundary slope of the spanning surface $S_A$ of $K$ is given by $s(S_A) = -2b -2c$. 
On the other hand, $s(S_P) = 0$. The difference in slopes of two surfaces is equal to the geometric intersection number, so we obtain that $i(\partial S_A, \partial S_P) = 2b + 2c$. An easy calculation shows that $\chi(S_A) = 1 - b -c$ and $\chi(S_P) = -1$. Using Theorem \ref{thm:GeneralMeridianBound} we have $\ell(\mu) \leq 3$. 

The same process will apply to any knot that admits an essential  state surface that  has non-zero slope. Large familes of such knots 
are the  semi-adequate knots or more generally the $\sigma$-adequate and $\sigma$-homogeneous knots \cite[Definition 2.22]{GutsBook}.
\end{example}
\vskip 0.03in

We now consider an application of Theorem \ref{thm:GeneralMeridianBound} to the case of adequate knots, and we derive Theorem \ref{thm:MeridianBound}
stated in the introduction. For the convenience of the reader, we restate the theorem.

\begin{named}{Theorem \ref{thm:MeridianBound}}{\emph {
Let $K$ be an adequate hyperbolic knot in $S^3$ with crossing number $c=c(K)$  and Turaev genus $g_T$. Let $C$ denote the maximal cusp of $S^3\setminus K$ and let
${\rm Area}(\partial C)$ denote the cusp area. Finally let
$\ell(\mu)$ and $\ell(\lambda)$ denote the length of the meridian and the  shortest $\lambda$-curve of $K$.
Then we have}

\begin{enumerate}
 \item $ \displaystyle \ell(\mu) \leq 3 + \frac{6g_T-6}{c} $
 \item $\ell(\lambda) \leq 3c + 6g_T - 6$
 \item ${\rm Area}(\partial C) \leq {9c} \left(1 + \dfrac{2g_T-2}{c}\right)^2$
\end{enumerate} }
\end{named}

\begin{proof}
Let $D$ be an adequate diagram for $K$ and let $S_A$ and
 $S_B$ be the corresponding all-$A$ and all-$B$ state surfaces respectively. By Theorem  \ref{essential},
$S_A$, $S_B$  are essential in $M=S^3\setminus K$.
Now $\partial S_A$ and $\partial S_B$ intersect transversely exactly twice per crossing in $D$. We show that this number of intersections is in fact minimal. To do so, we use the well-known ``bigon criterion" (see for example \cite[Proposition 1.7]{FarbMappingClassGroups}) which states that two transverse simple closed curves in a surface are in minimal position if and only if they do not form a bigon.

Consider the curves $\partial S_A$ and $\partial S_B$ near two consecutive crossings of $D$. If one crossing is an over-crossing and the other crossing is an under-crossing in the diagram $D$, then the intersection curves will be as in Figure \ref{fig:checkcrossing}. Note that this forms a diamond pattern on $\partial C$ near alternating crossings, hence there are no bigons near alternating crossings.
 \begin{figure}
\includegraphics{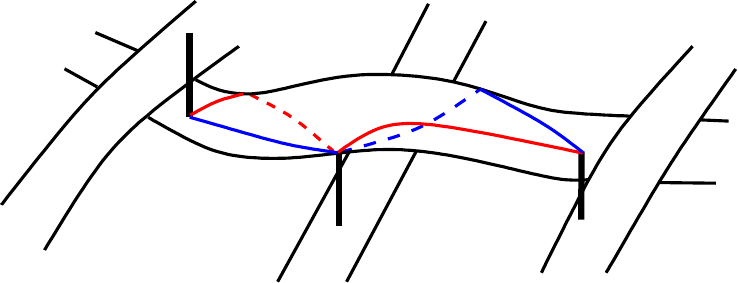}
\caption{The intersection of the surfaces $S_A$ (red) and $S_B$ (blue) with $\partial C$. Taken from \cite{EssentialTwisted}.}
\label{fig:checkcrossing}
\end{figure}

 Consider the Turaev surface $F(D)$ corresponding to $D$. Recall that $D$ is alternating on $F(D)$ and that $S_A, S_B$
are the checkerboard surfaces of this projection (Lemma \ref{turaevsurface}).

We turn to the case where two consecutive crossings in $D$ are over-crossings. The Turaev surface $T$ of $K$ in a neighborhood of these two crossings may be visualized as in Figure \ref{fig:TuraevAlternating}. The neighborhood may be straightened as shown in Figure \ref{fig:TuraevAlternating}, and we then see that the intersection of $\partial C$ with $S_A \cup S_B$ in a neighborhood of these two crossings is as in Figure \ref{fig:checkcrossing}. Therefore we get an intersection pattern similar to that of \ref{fig:checkcrossing} near pairs of consecutive over-crossings, and it follows that there are no bigons near pairs of over-crossings. Similarly there are no bigons near pairs of under-crossings. Thus we have $i(\partial S_A, \partial S_B)=2c$.

On the other hand, by construction of the state surface and using the notation of \S 3.2, we have $\chi(S_A)=v_A-c$ and $\chi(S_B)=v_B-c$. Note that if $\chi(S_A)=0$
or  $\chi(S_B)=0$ then $S_A$ or $S_B$ is a M\"obius band. But then $D$ is a diagram of the $(2, p)$ torus knot contradicting the assumption that $K$ is hyperbolic.
Thus $\chi(S_A), \chi(S_B)<0$.
Now by the definition of $g_D(T)$ and Theorem \ref{abe} we have

$$|\chi(S_A)| + |\chi(S_B)| =2c-v_A-v_B=  c+2g_T-2.$$

Using these observations, claims (1)-(3) of the statement follow immediately from Theorem  \ref{thm:GeneralMeridianBound}. We note that since  $ i(\partial S_A, \partial S_B)=2c $,
the coefficient 18 in the bound of the cusp area in  Theorem  \ref{thm:GeneralMeridianBound}, becomes 9 here. That is, we have 

$${\rm Area}(\partial C) \leq  18 \dfrac{( c+2g_T-2)^2}{2c}= {9c} \left(1 + \dfrac{2g_T-2}{c}\right)^2,$$
as claimed in the statement above.
\end{proof}

\vskip 0.03in

An immediate consequence of Theorem \ref{thm:MeridianBound} is that the meridian length of a knot with Turaev genus 1 never exceeds 3. Also as noted in Corollary \ref{finite}
for every Turaev genus there can be at most finitely many adequate knots where $\ell(\mu)\geq 4$.

The next result, stated in the introduction,  shows that in a certain sense ``most" adequate hyperbolic  knots  have meridian length less than 4.

 Before we state our result, we need bit of terminology.
A \emph{twist region} of knot diagram $D$ is a collection of bigons in $D$ that are adjacent end to end, such that there are no additional adjacent bigons on either end. A single crossing adjacent to no bigons is also a twist region. We require twist regions to be alternating, for if $D$ contains a bigon that is not alternating, then a Reidemeister move removes both crossings without altering the rest of the diagram. The number of distinct twist regions in a diagram $D$, denoted by $t=t(D)$,  is defined to be the\emph{ twist number} of that diagram.

\begin{named} {Theorem \ref{twist}} {\emph{ 
Let $K$ be a hyperbolic knot with an adequate diagram $D$ with $c$ crossings and $t$ twist regions. Then we have
$$\ell(\mu)\leq 3 +  \frac{3t}{c}- \frac{6}{c}.$$
 In particular if $c\geq 3t$ then we have $\ell(\mu) < 4$.}}

\end{named}
\begin{proof}
Let $g_T$ be the Turaev genus of $K$ and let $v_A$ and $v_B$ be the number of $A$ and $B$ state circles arising from $D$. Recall that $2g_T - 2 = c - v_A - v_B$. Now $v_A + v_B = v_{bi} + v_{nb}$ where $v_{bi}$ is the number of bigon regions in $D$ and $v_{nb}$ is the number of non-bigon regions. Then 
\begin{equation}
c - v_{bi} = t
\end{equation}
Since $D$ is adequate and hyperbolic, both the $A$ and $B$ resolutions must have a state circle corresponding to a non-bigon region. For if all the regions in one of the resolutions
are bigons then $D$ represents a $(2, p)$ torus knots, which is not hyperbolic. Therefore $v_{nb} \geq 2$ and it follows that 
\begin{align*}
2g_T -2 & = c - v_{bi} - v_{nb} = t - v_{nb} \leq t - 2
\end{align*}
Now by Theorem \ref{thm:MeridianBound} we see that
\begin{equation*}
\ell(\mu)< 3 + 3 \left( \frac{2g_T - 2}{c} \right)
\leq 3 + 3 \left( \frac{t-2}{c} \right)
\leq 3 + \frac{3t}{c} - \frac{6}{c}.
\end{equation*}

Now if $c\geq 3t$, say for example
if $D$ has at least three crossings per twist region, then  ${3t}/{c}\leq 1$, so we see that 
\begin{equation*}
\ell(\mu) < 3 + 1 - \frac{6}{c} < 4.
\end{equation*}

\end{proof}

Theorem \ref{twist} applies  to positive/negative closed braids. Let $B_n$ be the
braid group on $n$ strands, with $n \geq 3$, and let $\sigma_1,
\ldots, \sigma_{n-1}$ be the elementary braid generators. Let $b=\sigma_{i_1}^{r_1}\sigma_{i_2}^{r_2} \cdots \sigma_{i_k}^{r_k}$
be a braid in $B_n$.   It is straightforward to check that  if either $r_j \geq 2$ for all $j$, or
else $r_j \leq -2$ for all $j$, then  the braid closure $D_b$ of $b$ is an adequate diagram. In particular we have the following.

\begin{corollary} \label{braids} Suppose that a knot $K$ is represented by a braid closure $D_b$ 
such that either $r_j \geq 3$ for all $j$, or
else $r_j \leq -3$ for all $j$. 
Additionally, suppose $D_b$ is a
prime diagram. Then $K$ is
hyperbolic and the meridian length satisfies $\ell(\mu)<4$.
\end{corollary}
\begin{proof} The fact that $K$ is hyperbolic follows by \cite[Corollary 1.2]{FKP:hyp} and the claim about the meridian follows from Theorem \ref{twist}.

\end{proof}

\begin{remark}The twist number of any diagram of a hyperbolic knot $K$ bounds  ${\rm Area}(\partial C)$ from above. More precisely,  if a hyperbolic knot with maximal cusp $C$  admits a diagram with
$t$ twist regions then ${\rm Area}(\partial C)\leq 10 \sqrt 3 \cdot (t-1) \approx   17.32\cdot (t-1)$. The derivation of this bound is explained for example in  \cite{AdamsCuspSizeBounds}.
Note that if $c>>t$, this general bound does better than the one of Theorem  \ref{thm:MeridianBound}. On the other hand if $c=t$ and $g_T$ is small
the upper bound of  Theorem  \ref{thm:MeridianBound} is sharper than the general bound.  For instance  if $g_T\leq 1$ and $c=t$, then
 Theorem  \ref{thm:MeridianBound}  gives ${\rm Area}(\partial C)\leq  9 t$ which for $t\geq 3$ is sharper than the general bound.
\end{remark}

\begin{remark}\label{generalize} 
Theorem \ref{thm:GeneralMeridianBound} more generally applies to knots that admit alternating projections on surfaces so that they define essential checkerboard surfaces. Specifically, let
$F$ be closed surface that is embedded in $S^3$ in a standard or non-standard way. Let $K$ be a knot and suppose that there is a projection
$p: S^3 \longrightarrow F$ such that: (i) $p(K)$ is alternating and it separates $F$; (ii) the components of $F\setminus p(K)$ are disks that can be colored in two different colors so that the colors at each crossing of 
$p(K)$ meet in a checkerboard fashion; and (iii) the surface $F\setminus p(K)$ is essential in $S^3\setminus K$. 
For instance results similar to Theorem  \ref{thm:MeridianBound}
and Corollary \ref{finite} should also hold for 
 \emph{weakly alternating knots} considered by Ozawa \cite{Ozawa1} and further discussed in \cite{Howiethesis}. In this case
one  should replace $g_T$ with the genus of the surface $F$ and the crossing     number of the knot with the number of crossings of the alternating projection on $F$.

\end{remark}


\vskip 0.03in 

\section{Algorithm}\label{algorithm}


In this section we will finish the proof of Theorem \ref{meridiancriterion}. The  proof of the first part of the Theorem follows from part (a) of Theorem \ref{thm:GeneralMeridianBound}.
That is, if a hyperbolic knot $K$ in $S^3$ admits essential spanning surfaces $S_1, S_2$ such that

\begin{equation} \label{eq:MeridianCriterion}
|\chi(S_1)| + |\chi(S_2)| < \frac{b\cdot i(\partial S_1, \partial S_2) }{6},
\end{equation}
for some real number $b>0$,  then

$$\ell(\mu) < \dfrac{6(|\chi(S_1)| + |\chi(S_2)|)}{i(\partial S_1, \partial S_2)} < b.$$

\noindent The proof of Theorem \ref{meridiancriterion} will be complete once we show the following.

\begin{theorem}\label{thm:algorithm}
Given any hyperbolic knot $K$ and positive real number $b$, there is an algorithm which determines if there are spanning surfaces $S_1$ and $S_2$ satisfying inequality (\ref{eq:MeridianCriterion}). 
\end{theorem}

\smallskip


\begin{proof}
We now show that the condition of equation (\ref{eq:MeridianCriterion}) is algorithmically checkable. 
Start with a triangulation of the complement $M=S^3\setminus K$. There is an algorithm \cite{JacoTo} to turn the triangulation to one that has a single vertex that lies on the boundary of $M$.
Moreover, by Jaco and Sedgwick \cite{JacoDecisionProblems} there is an algorithm that ``layers"  this triangulation so that a meridian of $K$ is a single edge on $\partial M$ that is connected to the vertex of the triangulation. Call the latter  triangulation $\mathcal T$.
For normal surface background and  terminology the reader is referred to Matveev \cite{MatveevAlgorithmicBook} or the introduction of  \cite{JacoDecisionProblems}.
\smallskip

 \begin{lemma}\label{fundamental} Suppose that there are essential spanning surfaces $S_1, S_2$ that satisfy  (\ref{eq:MeridianCriterion}). Then we can find essential spanning surfaces that satisfy condition
 (\ref{eq:MeridianCriterion}) and, in addition, are normal fundamental surfaces with respect to  $\mathcal T$.
 
 \end{lemma}
 
  \begin{proof} 
 Suppose that one of $S_1, S_2$, say  $S_1$ is not connected. Then since $S_1$ is a spanning surface, and hence has a single boundary component, one of the connected components must be a closed surface $F$. Since $K$ is hyperbolic and $F$ is essential $\chi(F) \leq 0$, so taking $S = S_1 \backslash F$ we see that $|\chi(S)| \leq |\chi(S_1)|$, and $i(\partial S, \partial S_2) = i(\partial S_1, \partial S_2)$. Replacing $S_1$ with $S$, we may assume $S_1$ (and likewise $S_2$) is connected. 

Any essential surface in $S^3\backslash K$ may be isotoped to a normal surface with respect to   $\mathcal T$. Moreover, this normal surface may be  taken to be minimal in the sense of \cite[Definition 4.1.6]{MatveevAlgorithmicBook}. This means that the number of intersections of the surface with the edges of $\mathcal T$ is minimal in the (normal) isotopy class of the surface.
We will show that $S_1$ and $S_2$ may be taken to be \emph{fundamental} normal surfaces.

Suppose that $S_1$ is not fundamental. Then $S_1$ can be represented as a \emph{Haken sum} $S_1 = \Sigma_1 \oplus \hdots \oplus \Sigma_n \oplus F_1 \oplus \hdots \oplus F_k$ where each $\Sigma_i$ is a fundamental normal surface with boundary, and each $F_i$ is a closed fundamental normal surface. A theorem of Jaco and Sedgwick \cite{JacoDecisionProblems} states that each $\Sigma_i$ has the same slope. Since $S_1$ is a spanning surface,  and hence it has a single boundary component, this implies that $n = 1$. Since $K$ is hyperbolic, we know that either $\chi(F_i) < 0$ or $F_i$ is a boundary parallel torus for all $i$. In the latter case, it is known, as noted in \cite{HowieCharacterisation} that $\Sigma_1 \oplus F_i$ is isotopic in $S^3 \backslash N(K)$ to $\Sigma_1$. In the event that $\chi(F_i) < 0$, we note that $|\chi(\Sigma_1)| < |\chi(S_1)|$ and equation (\ref{eq:MeridianCriterion}) will hold with $S_1$ replaced by $\Sigma_1$. Moreover Matveev \cite[Corollary 4.1.37]{MatveevAlgorithmicBook} shows that $\Sigma_1$ must be incompressible. Therefore we can ignore the other terms of the Haken sum and assume that $S_1$ is fundamental. Similarly, we can assume that $S_2$ is fundamental.
\end{proof}

By Lemma \ref{fundamental},  in order to decide whether  there are spanning surfaces that satisfy  (\ref{eq:MeridianCriterion}), it is enough to decide whether there are fundamental
normal spanning surfaces with the same property.  Given $K$, there are only finitely many fundamental surfaces in $M$, and there  is an algorithm, due to Haken,  to  find them.
Let  $\mathcal{F}$ denote the list of all
fundamental surfaces. 
Since one of the boundary edges of the triangulation is a meridian, we may create  a subset $\mathcal{F_{\text{Span}}} \subset \mathcal{F}$ of fundamental normal  surfaces which are spanning
by finding the surfaces that intersect the meridian exactly once.
There is an algorithm to compute $\chi(F)$ for all surfaces $F \in \mathcal{F}$, and  to compute the minimal intersection number of two fundamental normal surfaces \cite{JacoTo}.
The algorithm now works by computing $|\chi(S_1)| + |\chi(S_2)|$ and $i(\partial S_1, \partial S_2)$ for all pairs of surfaces $S_1, S_2 \in \mathcal{F}_{\text{Span}}$ and checking whether inequality (\ref{eq:MeridianCriterion}) holds. If the condition holds, then use the algorithm of Haken to check that $S_1$ and $S_2$ are incompressible. If the condition fails for all pairs $S_1, S_2 \in \mathcal{F}_{\text{Span}}$, then inequality (\ref{eq:MeridianCriterion}) does not hold for any pair of essential spanning surfaces of $K$.

Knots with pairs of essential spanning surfaces $S_1, S_2$ with  $i(\partial S_1, \partial S_2)\neq 0$ are abundant. Note however  that not all knots have distinct essential spanning surfaces $S_1, S_2$ for which $i(\partial S_1, \partial S_2) \neq 0$. An example of such a knot is given by Dunfield  in \cite{nospanning}.  
 In this case, the algorithm outlined above will return that inequality (\ref{eq:MeridianCriterion}) cannot be satisfied. This may be seen as follows. In this case, either 
\begin{enumerate}
\item the set $\mathcal{F_{\text{Span}}}$ contains only one member, in which case there are no pairs for which to test, or 
\item the intersection number $i(\partial S_1, \partial S_2) = 0$ for all pairs $S_1, S_2 \in \mathcal{F}_{\text{Span}}$, and inequality (\ref{eq:MeridianCriterion}) will always fail since $K$ is hyperbolic implies $|\chi(S_1)| > 0$.
\end{enumerate} 
\end{proof}

\bibliographystyle{plain} \bibliography{biblio}

\end{document}